\documentclass{amsart}

\usepackage{amssymb,amsmath,enumerate}
\usepackage[colorlinks]{hyperref}

\usepackage[all]{xy}
\SelectTips{cm}{}

\newcommand{\da}[1]{\!\!\downarrow_{#1}}

\newcommand{\lra}{\longrightarrow}

\newcommand{\xla}{\xleftarrow}
\newcommand{\xra}{\xrightarrow}
\newcommand{\comesfrom}{{\leftarrow\hspace{-7pt}{\scriptstyle{\dashv}}}\,}

\newcommand{\ges}{\geqslant}

\newcommand{\dd}{\partial}

\newcommand\bsa{{\boldsymbol a}}

\newcommand\bsh{{\boldsymbol h}}

\newcommand\bsj{{\boldsymbol j}}

\newcommand\bss{{\boldsymbol s}}
\newcommand\bst{{\boldsymbol t}}
\newcommand\bsx{{\boldsymbol x}}
\newcommand{\BA}{{\mathbb A}}

\newcommand{\chr}{\operatorname{char}}
\newcommand{\Coker}{\operatorname{Coker}}

\newcommand{\depth}{\operatorname{depth}}

\newcommand{\eps}{\varepsilon}
\newcommand{\mcV}{\mathcal V}

\newcommand{\Tor}[4]{\operatorname{Tor}_{#1}^{#2}(#3,#4)}

\newcommand{\fm}{{\mathfrak m}}
\newcommand{\fn}{{\mathfrak n}}
\newcommand{\fp}{{\mathfrak p}}
\newcommand{\fq}{{\mathfrak q}}

\newcommand{\hh}[1]{\operatorname{H}(#1)}
\newcommand{\HH}[2]{\operatorname{H}_{#1}(#2)}

\newcommand\Image{\operatorname{Im}}

\newcommand\Ker{\operatorname{Ker}}

\newcommand{\ov}{\overline}

\newcommand{\wt}{\widetilde}

\newcommand{\fdim}{\operatorname{flat\,dim}}
\newcommand{\pdim}{\operatorname{proj\,dim}}
\newcommand{\rank}{\operatorname{rank}}
\newcommand{\rvar}[1]{V^{r}_{\!#1}}
\newcommand{\var}[1]{V_{\!#1}}
\newcommand{\shift}{{\sf\Sigma}}
\newcommand\Spec{\operatorname{Spec}}

\newtheorem{theorem}[subsection]{Theorem}
\newtheorem{proposition}[subsection]{Proposition}
\newtheorem{lemma}[subsection]{Lemma}
\newtheorem{corollary}[subsection]{Corollary}

\theoremstyle{definition}
\newtheorem{example}[subsection]{Example}

\theoremstyle{remark}
\newtheorem{remark}[subsection]{Remark}
\newtheorem{claim}{Claim}

\numberwithin{equation}{subsection}

\begin{document}

\title[Restricting homology to hypersurfaces]{Restricting homology to hypersurfaces}

\author[L.~L.~Avramov]{Luchezar L.~Avramov}
\address{Department of Mathematics,
University of Nebraska, Lincoln, NE 68588, U.S.A.}
\email{avramov@math.unl.edu}

\author[S.~B.~Iyengar]{Srikanth B.~Iyengar}
\address{Department of Mathematics,
University of Utah, Salt Lake City, UT 84112, U.S.A.}
\email{iyengar@math.utah.edu}

\thanks{Partly supported by NSF grants DMS-1103176 (LLA) and DMS-1700985 (SBI). We are grateful to Chris Drupieski for comments on earlier versions of this manuscript.}

\date{\today}


\begin{abstract} 
This paper concerns the homological properties of a module $M$ over a ring $R$  relative to a presentation $R\cong P/I$, where $P$ is local ring. It is proved that the Betti sequence of $M$ with respect to $P/(f)$ for a regular element $f$ in $I$ depends only on the class of $f$ in $I/\mathfrak{n} I$, where $\mathfrak{n}$ is the maximal ideal of $P$. Applications to the theory of supports sets in local algebra and in the modular representation theory of elementary abelian groups are presented.
\end{abstract}

\dedicatory{To Dave Benson on the occasion of his 60th birthdays.}

\maketitle
   
\section*{Introduction}
\label{sec:intro}
This work concerns an analogue for commutative rings of Carlson's theory of rank varieties for elementary abelian groups~\cite{Ca}.
It takes the following form: Given a (noetherian, commutative) local ring $R$ that is a quotient of a local ring $P$, so that $R:=P/I$ for an ideal $I$, the goal is to study an $R$-module $M$ by its restrictions to hypersurfaces $P/(f)$ where $f\in I$ is a regular element (that is to say, not a zero divisor). The rationale is that homological algebra over such hypersurfaces is well-understood, especially when $P$ is a regular ring; then one can, for example, take recourse to the theory of matrix factorizations initiated by Eisenbud~\cite{Ei}.

In this endeavor a basic question is how the properties of $M$ change as we vary the element $f$. The result below, contained in Theorem~\ref{thm:hypersurfaces}, addresses this issue.

\medskip

\emph{
If $f,g$ in $I$ are regular elements with $f-g$ in $\fn I$, where $\fn$ is the maximal ideal of $P$, then for any $R$-module $M$ there are isomorphisms of $k$-vector spaces}
\[
\Tor{i}{P/(f)}k{M} \cong \Tor{i} {P/(g)}k{M} \quad\text{for each $i\ge 0$.}
\]
It follows that when $M$ is finitely generated, its projective dimension as a module over $P/(f)$ and over $P/(g)$ are simultaneously finite. This latter result was proved by Avramov~\cite{Av:vpd} when $P$ is regular and $I$ is generated by a regular sequence; it is part of the theory of cohomological support varieties for modules over complete intersections~\cite{AB}. Jorgensen~\cite{Jo} generalized this to any ideal $I$ in a domain $P$.

The isomorphism above yields more: \emph{the Betti numbers of  $M$ over $P/(f)$ and over $P/(g)$ are equal}.  Moreover the statement carries over to the context of graded rings and implies that the graded Betti numbers and hence also invariants derived from them, like  regularity, are equal.

That said, the motivation for writing this manuscript is not so much the greater generality of the result; rather, it is to give an alternative point of view---one that lays bare the structural reason behind the numerical coincidences. The proof, given in Section~\ref{sec:Hypersurfaces}, uses a modicum of Differential Graded (henceforth abbreviated to DG) homological algebra, recapped in Section~\ref{sec:DGA}.

As one application of the theorem, in Section~\ref{sec:groups} we describe how to deduce from it certain results of Carlson~\cite{Ca}, Friedlander and Pevtsova~\cite{FP}, and Suslin~\cite{Su}, that underlie the theory of rank varieties for finite groups and group schemes.

The theorem above also leads to the notion of a support set of $M$ with respect to the homomorphism $\pi\colon P\to R$, denoted $\mcV_{\pi}(M)$. It is a subset of $k^{c}$, where $c$ is the minimal number of elements required to generate the ideal $I$. We prove:

\medskip

\emph{
When $I$ contains a regular element and $M$ is finitely generated, the subset $\mcV_{\pi}(M)\subseteq {k}^{c}$ is closed in the Zariski topology.
}

\medskip

This result is contained in Theorem~\ref{thm:closed} which identifies the support set of $M$ as the algebraic set defined by the annihilator of a certain module, over a polynomial ring $k[s_{1},\dots,s_{c}]$, naturally associated to $M$.  Up to radical this annihilator ideal can be described explicitly, as we illustrate in Section~\ref{sec:equations}. 
 
\section{Differential graded algebra}
\label{sec:DGA}
 In this section we recall what little is needed, by way of constructs and results concerning DG algebras and DG modules. Our reference for this material is \cite{Av:barca}.

Let $P$ be a commutative ring and $A$  a DG $P$-algebra; it will be implicit that $A$ is graded-commutative and satisfies $A_{i}=0$ for $i<0$. 

\subsection{Tensor products}
  \label{ssec:DSAs}
When $A$ and $B$ are DG $P$-algebras, so is $A\otimes_PB$, with standard
differential and products 
\[
(a\otimes b)(a'\otimes b')=(-1)^{|a'||b|}aa'\otimes bb'\,.
\]
We identify $A$ and $B$ with their images in $A\otimes_P B$ and write $ab$ 
instead of $a\otimes b$. 

\subsection{Divided powers}
  \label{ssec:divided}
For $w\in A_{2d}$ with $d\ge1$, we say that $(w^{(i)}\in A_{2di})_{i\ges 0}$ is 
a \emph{sequence of divided powers} of $w$ if $w^{(0)}=1$, $w^{(1)}=w$, and there are equalities
  \[
w^{(i)} w^{(j)} = \frac{(i+j)!}{i! j!}\, w^{(i+j)}
   \quad\text{and}\quad
\dd(w^{(i)}) = \dd(w) w^{(i-1)} \quad\text{for all}\quad i,j\ge 0\,.
  \]

Induction on $i$ yields $w^i=i!w^{(i)}$ for $i\ge0$. Thus $w$ has a unique sequence of divided powers 
if $\chr(k)=0$, or if $\chr(k)=p>0$ and $A_j=0$ for $j\ge 2dp-1$.  However, not every element of even 
positive degree has divided powers in general.

If $v$ and $w$ have divided powers, then so do $aw$, for $a\in P$, and $v+w$ with
\[
(a w)^{i} = a^{i}w^{(i)}\quad \text{and}\quad (v+w)^{(h)} = \sum_{i+j=h}v^{(i)}w^{(j)}\,.
\]

\subsection{The Koszul complex}
\label{ssec:koszul}
Let $t_{1},\dots,t_{n}$ be elements in $P$ and $K$ the Koszul complex on $\bst$. Thus $K$ is a DG $P$-algebra with underlying graded algebra the exterior algebra $\bigwedge_{P}K_{1}$, where $K_{1}$ is a free $P$-module with basis $x_1,\dots,x_n$, and differential $\dd$ defined by the condition $\dd(x_i)=t_i$.  Then $\HH 0K=P/(\bst)$,  so $K$ comes equipped with a canonical morphism of DG $P$-algebras  $\varepsilon\colon K\to P/(\bst)$.  

Every element of $K_{2d}$ with $d\ge1$ has a sequence of divided powers.  Only those for $d=1$ are needed here, and we proceed to define them ad hoc.

Write each $w\in K_{2}$ as $w=\sum_{1\le a< b\le n}r_{a,b}\,x_ax_b$ with $r_{a,b}\in P$ and set
  \[
w^{(i)}:= \sum_{\substack{{1\le a_j<b_j\le n} \\ {1 \le j\le i}}} 
(r_{a_1,b_1}\cdots r_{a_i,b_i})\,x_{a_1}x_{b_1}\cdots x_{a_i}x_{b_i} \quad\text{for}\quad i\ge 1\,.
   \]
The right-hand side is well defined because the expression of $w$ is unique.  A direct 
computations shows that $\dd(w^{(i)}) = \dd(w) w^{(i-1)} $ holds for $i\ge0$.

\subsection{The Tate construction}
\label{ssec:tate}
Let $P\langle y\rangle$ be a graded free $P$-module with basis $\{y^{(i)}\}_{i\ges 0}$, where $|y^{(i)}|=2i$,  
and define a $P$-linear multiplication table by
  \[
y^{(0)}=1
\quad\text{and}\quad
y^{(i)} y^{(j)} = \frac{(i+j)!}{i! j!}\, y^{(i+j)} 
\quad \text{for all}\quad i,j\ge 0\,.
  \]
By construction, the element $y$ has a sequence of divided powers, namely, $\{y^{(i)}\}$.

Let  $A$ be a DG $P$-algebra and $z\in A_{1}$ a cycle. We write $A\langle y\mid \dd(y)=z\rangle$ for the DG $P$-algebra with underlying graded algebra $A\otimes_{P} P\langle y\rangle$, with differential extending the one on $A$ and satisfying
\[
\dd(y^{(i)}) = zy^{(i-1)} \quad \text{for all $i\ge 1$}
\]
We abbreviate $A\langle y\mid \dd(y)=z\rangle$ to $A\langle y\rangle$ if the differential on $y$ is clear, or irrelevant.  

If an element $a\in A_{2}$ has a sequence of divided powers, then so do elements of the $P$-submodule of $A\langle y\rangle$ generated by $a$ and $y$; see \ref{ssec:divided}.

The map $A\to A\langle y\rangle $ assigning $a\in A$ to $ay^{(0)}$ is one-to-one and a morphism of DG $P$-algebras. For $i\ge 0$, it induces a homomorphism 
\[
\HH i{A}\lra \HH i{A\langle y\rangle}
\]
that is bijective for $i=0$ and surjective for $i=1$, with kernel the ideal generated by the class of $z$. The result below and its proof are standard; see \cite[\S6]{Av:barca}. Details of the proof are given for ease of reference. 

\begin{lemma}
\label{lem:acyclic}
Let $z$ be a cycle in $A_{1}$.
\begin{enumerate}[\quad\rm(1)]
\item
If an element $w\in A_{2}$ has a sequence of divided powers, then the DG $P$-algebras $A\langle y\mid \dd(y)=z\rangle$ and $A\langle y\mid \dd(y)=z+\dd(w)\rangle$ are isomorphic.
\item
If the class of $z$ is a basis for the $\HH 0A$-module $\HH 1A$ and $\HH iA=0$ for $i\ge 2$, then the canonical map
\[
A\langle y\mid \dd(y)=z\rangle \lra \HH 0A
\]
is a quasi-isomorphism.
\end{enumerate}
\end{lemma}

\begin{proof}
It is readily verified that the $A$-linear map
\[
A\langle y\mid \dd(y)=z\rangle \lra A\langle y\mid \dd(y)=z+\dd(w)\rangle
\]
that assigns $y^{(i)}$ to $(y+w)^{(i)}$ is a morphism of DG $P$-algebras; it is bijective, with inverse the $A$-linear map that assigns 
$y^{(i)}$ to $(y-w)^{(i)}$. This justifies (1).

The $A$-linear map $\Theta\colon A\langle y\rangle \to \shift^{2}A\langle y\rangle$  assigning $y^{(i)}$ to $y^{(i-1)}$ is a morphism of DG $A$-modules, with kernel $A$, identified as the DG $A$-submodule $Ay^{(0)}$ of $A\langle y\rangle$. Thus, there is an exact sequence of DG $A$-modules
\[
0\lra A\lra A\langle y\rangle \xra{\ \Theta\ } \shift^{2} A\langle y\rangle\lra 0\,.
\]
In homology, this yields that $\HH 0{A\langle y\rangle}\cong \HH 0{A}$,  an exact sequence
\[
0\lra \HH 2{A\langle y\rangle}\lra \HH 0{A\langle y\rangle}\xra{\ \eth\ }  \HH 1A \lra \HH 1{A\langle y\rangle}\lra 0
\]
 of $\HH 0A$-modules where $\eth$ maps the class of $1$ to the class of $z$, and isomorphisms
\[
\HH i{A\langle y\rangle}\cong \HH {i-2}{A\langle y\rangle}\quad\text{for $i\ge 3$.}
\]
The hypothesis yields that $\eth$ is an isomorphism, so $\HH 2{A\langle y\rangle} = 0 = \HH 1{A\langle y\rangle}$ and then the isomorphisms above entail $\HH i{A\langle y\rangle}=0$ for all $i\ge 1$.
\end{proof}

\subsection{Acyclic closures}
\label{ssec:closures}
Let $\bst$ be a finite set of elements generating an ideal $\fn$ of $P$ and  $K$ the Koszul complex on $\bst$.  Let $z_{1},\dots,z_{c}$ be cycles in $K_{1}$ that generate the $\HH 0K$-module $\HH 1K$. Iterating the construction from \ref{ssec:tate} yields a DG $P$-algebra
\[
T:=K\langle y_{1},\dots,y_{c}\mid \dd(y_{i})=z_{i}\rangle\,.
\]
By construction $\HH 0T=P/\fn$ and $\HH 1T=0$. A repetition of this procedure, involving also the adjunction of exterior variables leads to a DG $P$-algebra $A$ containing $T$ as a DG subalgebra and satisfying $\HH iA=P/\fn$ and $\HH iA=0$ for $i\ge 1$. This is an \emph{acyclic closure} of $P/\fn$; see \cite[Chapter 6]{Av:barca}.

As noted in \ref{ssec:koszul},  in the DG $P$-algebra $K$ each element of even positive degree has a system of divided powers; the ones on $K_{2}$ were described explicitly. These induce a sequence of divided powers on the elements of degree two in $T$, see \ref{ssec:tate}, and hence also in $A$.

\section{Hypersurfaces}
\label{sec:Hypersurfaces}
Let $P$ be a commutative noetherian ring, $\fn$ a maximal ideal of $P$; set $k=P/\fn$.  We fix an ideal $I$ of $P$ contained in $\fn$ and set $R=P/I$. For every $f\in I$, let 
\[
\pi_{f}\colon  P/{(f)} \lra R
\]
be the canonical  homomorphism of rings. 

\begin{theorem}
\label{thm:hypersurfaces}
Let $I\subset P$ be an ideal containing a regular element, set $R:=P/I$ and let $M$ be an $R$-complex.
\begin{enumerate}[\quad\rm(1)]
\item
For each $f\in I$ there exists an $h\in \fn I$ such that $f+h$ is $P$-regular.
\item
Let $f$ and $g$ be regular elements lying in $I$.  If $f-g$ is in $\fn I$, then there is an isomorphism of graded $k$-vector spaces
\[
\Tor{}{P/(f)}k{M} \cong \Tor{} {P/(g)}k{M} \,.
\]
\item
If $f$ is a regular element lying in $\fn I$, then there is an isomorphism of graded $k$-vector spaces
\[
\Tor{}{P/(f)}k{M} \cong \Tor{}PkM \otimes_{k}k\langle y\rangle\,,\quad\text{with $|y|=2$.}
\]
\end{enumerate}
\end{theorem}

The proof is given at the end of this section.  Part (1) is well-known and part (3) is essentially contained in the work of Shamash~\cite[Theorem~1]{Sh}.

Throughout the rest of the section $\bst:=t_{1},\dots,t_{n}$ denotes a finite generating set for the ideal $\fn$ and $A$ denotes an acyclic closure of $k$ over $P$; see \ref{ssec:closures}.

\subsection{Degree one cycles} 
  \label{ssec:Z}
Set $C:= A\otimes_{P}R$; this is a DG $R$-algebra. As $A$ is a free resolution of $k$, there is an isomorphism of graded $k$-spaces
  \[
\hh{C} \cong \Tor{}PkR\,.
  \]
Expressing elements of $I$ as linear combinations of  $t_1,\dots,t_n$ gives rise to degree one cycles of $C$. 
We give a compact description of that process by using the $P$-module
 \[
Z:= \{(f,x)\in I \oplus A_{1} \mid f = \dd(x)\}\,;
  \]
this is the fiber product of the maps $I\hookrightarrow P \gets A_{1}:\partial$.
 
\begin{lemma}
\label{lem:Z}
The assignments $f\comesfrom\hskip-1pt(f,x)\mapsto x\otimes 1$ define surjective $P$-linear maps
   \begin{gather*}
I \xla{\ \eps\ } Z\xra{\ \zeta\ }Z_{1}(C) \,.
   \end{gather*}

An element $z\in Z$ satisfies  $\eps(z) \in \fn I$ if and only if $\zeta(z)$ is in $\dd_{2}(C)$.
\end{lemma}

\begin{proof}
Let $x$ be an arbitrary element of $A_{1}$.

The element $x\otimes 1$ is a cycle in $C_1$ if and only if $\dd(x)$ lies in~$I$.  It follows that $\zeta$ is well-defined and that both $\zeta$ and $\eps$ are surjective.  

Fix $z:=(f,x)$ in $Z$, so $\zeta(z)=x\otimes 1$. Then $\zeta$ is in $\dd(C)$ if and only if $x$ lies in $\dd(A)+IA_1$.  Thus $\zeta(z)$ in $\dd(C)$ yields $f = \dd(x) \in \dd(IA_{1}) = \fn I$.  

Conversely, $f$ in $\fn I$ gives $f = \dd(\sum a_{i}x_{i})$ with $\{x_{i}\}_{i=1}^{n}$ 
in $A_{1}$ and $\{a_{i}\}_{i=1}^{n}$ in $I$, so that $x-\sum a_{i}x_{i}$ is a cycle in~$A_1$.  As $\HH 1A$ is zero, we get
$x-\sum a_{i}x_{i}=\dd(w)$ for some $w\in A$, whence $x\otimes1=\dd(w\otimes1)$.
\end{proof}

\subsection{Hypersurfaces}
\label{ssec:hyper}
Fix an element $f$ in $I$, set $\ov P:=P/(f)P$ and consider the DG $\ov P$-algebra $\ov A:=A\otimes_{P}\ov P= A/fA$. Choose an element $z_{f} \in A_{1}$ such that $\dd(z_{f})=f$.  The residue class of $z_{f}$ in $\ov A$ is a cycle, denoted $\ov {z}_{f}$. Set
\[
B_{f}:= \ov A\langle y\mid \dd(y) = \ov {z}_{f}\rangle\,.
\]
This has a canonical augmentation $B_{f}\to k$, which is a morphism of DG $\ov A$-algebras.

\begin{lemma}
If $f$ is a regular element, then $B_{f}\to k$ is a quasi-isomorphism.
\end{lemma}

\begin{proof}
Applying Lemma~\ref{lem:Z} with $(f)$ in place of $I$, we see that $\HH1{\ov A}$ is generated by the class of $z_{f}$ and is  zero if and only if $f$ lies in $\fn f$; the latter possibility is ruled out by the hypothesis that $f$ is regular. By using the resolution 
\[
0\to P\xra{f} P\to 0\,,
\]
of $\ov P$ over $P$ we get 
\[
\Tor{i}Pk{\ov P}\cong 
\begin{cases}
k & \text{for $i=0,1$} \\
0 &\text{otherwise}
\end{cases}
\]
Since $\hh{\ov A} \cong \Tor{}Pk{\ov P}$, it remains to apply Lemma~\ref{lem:acyclic}
\end{proof}

\begin{proof}[Proof of Theorem~\ref{thm:hypersurfaces}]
The zero-divisors in $P$ are the elements of its associated primes; call them $\fp_{1},\dots,\fp_{n}$. Since $I$ has a regular element so does $\fn I$, and hence the latter is not contained in $\cup_{i=1}^{n}\fp_{i}$. Thus, (1) follows from \cite[Theorem~124]{Kap}.

In the next steps the notation and constructions from \ref{ssec:Z} and \ref{ssec:hyper} will be used. In particular, $A$ is a DG $P$-algebra resolution of $k$, and $C:=A\otimes_{P}R$. Moreover $B_{f}$ is a DG $\ov P$-algebra resolution of $k$.

\begin{claim}
\label{cl:homology}
$\Tor{}{P/(f)}k{M}$ is the homology of the DG $R$-module
\[
C\langle y\mid \dd(y)=z_{f}\otimes 1\rangle \otimes_{R}M
\]

Indeed $\Tor{}{P/(f)}k{M}$ is the homology of the DG $R$-module $B_{f}\otimes_{P/(f)} M$. Associativity of tensor products yields isomorphisms of $R$-complexes
\begin{align*}
B_{f}\otimes_{P/(f)} M
	&\cong ((A/fA)\langle y\mid \dd(y) = \ov{z}_{f}\rangle \otimes_{P/(f)} R) \otimes_{R}M \\
	&\cong C\langle y\mid \dd(y) =  z_{f}\otimes 1 \rangle \otimes_{R}M
\end{align*}
This justifies the claim.
\end{claim}

(3) When $f\in \fn I$ holds, the element $z_{f}$ is in $I A$ and hence its residue class in $C$ is zero. The preceding claim then yields
\begin{align*}
\Tor{}{P/(f)}k{M} 
	&\cong \hh{C\langle y\mid \dd(y)=0\rangle \otimes_{R}M} \\
	&= \hh{C\otimes_{R}M}\otimes_{R}R\langle y\rangle \\ 
	&=\hh{A\otimes_{P}M}\otimes_{R}R\langle y\rangle \\ 
	&\cong \Tor{}PkM\otimes_{R}R\langle y\rangle \\
	&\cong \Tor{}PkM\otimes_{k}k\langle y\rangle 
\end{align*}

(2) Recall that the element $g$ is also a regular element in $I$, with $f-g$ in $\fn I$.

\begin{claim}
\label{cl:iso}
There is an isomorphism of DG $R$-algebras
\[
C\langle y\mid \dd(y)=z_{f}\otimes 1\rangle \cong C\langle y\mid \dd(y)=z_{g}\otimes 1\rangle\,.
\]
\end{claim}

Indeed, Lemma \ref{lem:Z} provides elements $(f,x)$ and $(g,y)$ in the module $Z$ defined in \S\ref{ssec:Z}, and $w$ in $C_2$ 
satisfying $x\otimes 1 - y\otimes 1=\dd(w)$.  Since $w$ has a sequence of divided powers (see \S\ref{ssec:closures}),  
Lemma~\ref{ssec:tate} yields the desired isomorphism. 

The isomorphism in the Claim \ref{cl:iso} induces an isomorphism of $R$-complexes
\[
C\langle y\mid \dd(y)=z_{f} \otimes 1 \rangle\otimes_{R}M \cong C\langle y\mid \dd(y)=z_{g} \otimes 1 \rangle\otimes_{R}M\,.
\]
Taking homology, and recalling Claim \ref{cl:homology}, yields the desired isomorphism.
\end{proof}

\section{Support sets}
\label{sec:supportsets}
 
Let $(P,\fn,k)$ be a (commutative noetherian) local ring, with maximal ideal $\fn$  and residue field $k$. Let $\pi\colon P\to R$ be a surjective homomorphism of rings such that the ideal $I:=\Ker(\pi)$ contains a regular element.  Throughout, $M$ will be an $R$-module; the results presented below all carry over to complexes; see Remark~\ref{rem:complexes}. 

\subsection{Support sets}
\label{ss:support}
We view the $k$-vector space $I/\fn I$ as endowed with the Zariski topology. Given $f\in I$ we write $[f]$ for its residue class in $I/\fn I$.  The \emph{homological support set} of $M$ with respect to $\pi$ is the subset of $I/\fn I$ described by
\[
\mcV_{\pi}(M):= \left\{ [f]\in I/\fn I \left| 
\begin{gathered}
\text{there exists a regular element $g\in I$ with } \\ 
\text{$f-g\in \fn I$ such that $\Tor i{P/(g)}kM\ne 0$}\\
\text{ for infinitely many integers $i$}
\end{gathered}
\right.\right\}\cup \{0\}\,.
\]
Thanks to Theorem~\ref{thm:hypersurfaces}, one can test whether $[f]\in \mcV_{\pi}(M)$ holds by considering $\Tor{}{P/(g)}kM$ for  \emph{any} regular element $g$ in $I$ with $[f]=[g]$. Evidently, if $[f]$ is in $\mcV_{\pi}(M)$, then so is $[\lambda f]$ for any non-zero element $\lambda$ in $k$. Hence $\mcV_{\pi}(M)$ is an affine cone with vertex at $\{0\}$. Therefore one could also consider the homological support set as a subset of the projective space associated to $I/\fn I$. See Remark~\ref{rem:history} for antecedents of this construction.

The condition on the vanishing of Tor appearing in the definition of the homological support set of $M$ has a more familiar interpretation, at least under additional conditions on the module $M$. This is explained in the next paragraph.

\subsection{Projective dimension}
\label{ss:pdim}
We write $\pdim_{R}M$ for the projective dimension of the $R$-module $M$. If $M$ is finitely generated, then
\begin{equation}
\label{eq:perfect}
\pdim_{R}M <\infty \iff \Tor iRkM=0 \quad \text{for $i\gg 0$}\,;
\end{equation}
see, for example, \cite[Proposition 5.5(P)]{AF}.

\subsection{Alternative description of support}
\label{ssec:alternative}
Fix a minimal generating set $f_{1},\dots,f_{c}$ for the ideal $I$. The residue classes $[f_{1}], \dots, [f_{c}]$ form a basis for $I/\fn I$ as a $k$-vector space. We will use this to identify $I/\fn I$ with $k^{c}$ whenever necessary, or convenient.

Given a point $\bsa:=(a_{1},\dots,a_{c})$ in $k^{c}$ it follows from Theorem~\ref{thm:hypersurfaces}(1) that there is an element $\bsa':=(a'_{1},\dots, a'_{c})$ in $P^{c}$ such that 
\begin{enumerate}
\item
$a'_{i} = a_{i} \mod \fn$ for $i=1,\dots,c$;
\item
$\sum a'_{i}f_{i}$ is regular in $P$.
\end{enumerate}
We will call such an element $\bsa'$ a \emph{lifting} of $\bsa$ to $P$. When $M$ is finitely generated, one can also describe the subset $\mcV_{\pi}(M)$ from \ref{ss:support} as
\[
\mcV_{\pi}(M)=\left\{\bsa\in {k}^{c} 
\left| 
\begin{gathered}
\text{there exists a lifting ${\bsa}'$ of $\bsa$ with} \\
\pdim_{P/(f)}(M)=\infty \text{ for $f=\sum a'_{i}f_{i}$}
\end{gathered}
\right.\right\}
\]
In particular, the right hand side does not depend on the choice of a minimal generating set for $I$. The subset $\mcV_{\pi}(M)\subseteq k^{c}$ is closed when $M$ is finitely generated; this is contained in Theorem~\ref{thm:closed}. For the proof of this result, we require a statement about complexes over regular rings, presented below.

\subsection{Complexes over regular rings} By a \emph{regular ring} we mean a commutative noetherian ring $S$  such that the local rings $S_{\fp}$ are regular for all $\fp$ in $\Spec S$.  For any $\fp\in \Spec S$, We write $k(\fp)$ for the residue field, $(S/\fp)_{\fp}$, at $\fp$.

 \begin{lemma}
 \label{lem:closed}
Let $S$ be a regular ring and let $X$ be a complex of  $S$-modules such that $\HH iX$ is finitely generated for each $i$ and equals $0$ for $i\ll 0$. Fix a prime ideal $\fp$ in $S$. The following conditions are equivalent.
 \begin{enumerate}[\quad\rm(1)]
 \item
 $\Tor i S{k(\fp)}X=0$ for $i\gg 0$;
 \item
 $k(\fp) \otimes_{S} {\HH iX}=0$ for $i\gg 0$;
 \item
${\HH iX}_{\fp}=0$ for $i\gg 0$.
 \end{enumerate}
 \end{lemma}
 
 \begin{proof}
Since the action $S$ on $k(\fp)$ factors through the localization homomorphism $S\to S_{\fp}$ and $S_{\fp}$ is flat as an $S$-module, one has isomorphisms
\begin{align*}
\Tor {}S{k(\fp)}X  &\cong \Tor {}{S_{\fp}}{k(\fp)}{X_{\fp}} \\
k(\fp)\otimes_{S} {\HH iX} &\cong k(\fp) \otimes_{S_{\fp}} {\HH iX}_{\fp}
\end{align*}
Moreover, $X_{\fp}$ is a complex of $S_{\fp}$-modules with $(X_{\fp})_{i}=0$ for $i\ll 0$. Thus, replacing $S$ and $X$ by their localizations at $\fp$, we assume  $S$ is local and  $\fp$ is its maximal ideal.
 
 (1)$\iff$(3) Let $E$ be the Koszul complex on a finite set of generators of the ideal $\fm$. Since $E$ is in the thick subcategory of the derived category of $S$ generated by $k$, the hypotheses in (1) implies $\HH i{E\otimes_{S}X}=0$ for $i\gg 0$.  Using the fact that the $S$-modules $\HH iX$ are finitely generated, a standard argument now implies $\HH iX=0$ for $i\gg 0$; see, for example, \cite[1.3]{FI}.
 
 (2)$\iff$(3) This is by Nakayama's Lemma.
\end{proof}

\subsection{On being closed}
\label{ssec:closed}
Let $P[\bss]$ be the polynomial ring over $P$ in indeterminates $\bss:=s_{1},\dots,s_{c}$ and 
\[
\wt Q:= P[\bss]/(\wt f) \quad\text{where $\wt f:=\sum_{i=1}^{c} s_{i}f_{i}$.}
\]
The ideal $I$ contains a $P$-regular element so there is no non-zero element in $P$ that annihilates each of $f_{1},\dots,f_{c}$. Therefore $\wt f$ is a regular element in $P[\bss]$; see, for example, \cite[Chapter 1, Exercise 2]{AM}. Thus the complex of $P[\bss]$-modules
\begin{equation}
\label{eq:resQdot}
0\lra P[\bss] \xra{\ \wt f\ } P[\bss]\lra 0
\end{equation}
is a free resolution of $\wt Q$.

For each ${\bsa}':=(a'_{1},\dots,a'_{c})$ in $P^{c}$ form the homomorphism of $P$-algebras
\[
\eps_{{\bsa}'}\colon P[\bss]\lra P \quad\text{given by $s_{i}\mapsto a'_{i}$.}
\]
In what follows, we write $P_{\bsa'}$ for $P$ viewed as a $P[\bss]$-module via $\eps_{{\bsa}'}$. Base change along the canonical surjection $P[\bss]\to \wt Q$ gives a commutative square 
\[
\begin{gathered}
\xymatrixcolsep{1pc}
\xymatrix{
P[\bss] \ar@{->}[d] \ar@{->}[rr]^{\eps_{{\bsa}'}} && P \ar@{->}[d]& \\
\wt Q \ar@{->}[rr] && Q \ar@{=}[r] & P/(f)
}
\end{gathered}
\]
where $f=\sum a'_{i}f_{i}$.

\begin{lemma}
\label{lem:pushout}
When the element $f:=\sum a'_{i}f_{i}$ is $P$-regular, one has
\[
\Tor i{P[\bss]}{P_{\bsa'}}{\wt Q} =
\begin{cases}
Q  & \text{for $i=0$} \\
0   & \text{otherwise}
\end{cases}
\]
\end{lemma}

\begin{proof}
Using the free resolution~\eqref{eq:resQdot} of $\wt Q$ over $P[\bss]$, one gets that $\Tor{}{P[\bss]}{P_{\bsa'}}{\wt Q}$ is the homology of the complex
\[
0\lra P\xra{\  f\ } P \lra 0\qedhere
\]
\end{proof}

Set $M[\bss]:= P[\bss]\otimes_{P}M$. Evidently $f\cdot M[\bss]=0$ so $M[\bss]$ is  an $\wt Q$-module.

\begin{proposition}
\label{prp:mx}
With notation as above, the following statements hold.
\begin{enumerate}[\quad\rm(1)]
\item
When $p:=\pdim_{P}M$ is finite, for any $\wt Q$-module $N$ one has
\[
\Tor {i+2}{\wt Q}N{M[\bss]} \cong \Tor {i}{\wt Q}N{M[\bss]} \quad\text{for $i\ge p$.}
\]
\item
When $\sum_{i} a'_{i}f_{i}$ is regular in $P$, one has
\[
\Tor i{\wt Q}Q{M[\bss]} =
\begin{cases}
M & \text{for $i=0$}\\
0 & \text{otherwise}
\end{cases}
\]
\end{enumerate}
\end{proposition}

\begin{proof}
The argument uses the standard first quadrant change of rings spectral sequence associated to the homomorphism $P[\bss]\to\wt Q$, and a  $P[\bss]$-module $L$.
\[
E^{2}_{i,j}:= \Tor i{\wt Q}{\Tor j{P[\bss]}L{\wt Q}}{M[\bss]} \Longrightarrow \Tor {i+j}{P[\bss]}L{M[\bss]} 
\]

(1) Set $L:=N$ in the spectral sequence. From \eqref{eq:resQdot} and $\wt f N=0$ one gets that
\[
\Tor i{P[\bss]}N{\wt Q} \cong 
\begin{cases}
N & \text{for $i=0,1$} \\
0 & \text{otherwise}
\end{cases}
\]
Thus the $E^{2}_{i,j}=0$ for $q\ne 0,1$ so the spectral sequence unwinds to an exact sequence
\begin{align*}
 \cdots \to & \Tor {i+2}{P[\bss]}N{M[\bss]} \to \Tor {i+2}{\wt Q}N{M[\bss]} \to \Tor i{\wt Q}N{M[\bss]}  \\
    \to   & \Tor {i+1}{P[\bss]}N{M[\bss]} \to\cdots
\end{align*}
Since $\pdim_{P}M=p$, one has  $\pdim_{P[\bss]}{M[\bss]}=p$, so the first and the last modules in the sequence above are zero
for $i\ge p$. This yields (1).

(2)  Consider the spectral sequence above with $L:=P_{\bsa'}$. In view of Lemma~\ref{lem:pushout}, this spectral sequence collapses on the second page and yields isomorphisms
\[
\Tor i{\wt Q}Q{M[\bss]} \cong \Tor {i}{P[\bss]}{P_{\bsa'}} {M[\bss]} \quad\text{for each $i$.}
\]
Let $F$ be a free resolution of $M$ as a $P$-module; then $F[\bss]:=P[\bss]\otimes_{P}F$ is a free resolution of $M[\bss]$ over $P[\bss]$. Associativity of tensor products yields 
\[
P_{\bsa'} \otimes_{P[\bss]} F[\bss] \xra{ \simeq} P \otimes_{P} M \cong M
\]
 of $P$-complexes. This yields 
\[
\Tor i{P[\bss]}{P_{\bsa'}} {M[\bss]} =
\begin{cases}
M  & \text{for $i=0$} \\
0   & \text{otherwise}
\end{cases}
\]
The preceding computations justifies the assertion in (2).
\end{proof}

The statement and proof of the next result involve  constructions introduced in \ref{ssec:closed}. We identify $k[\bss]$ and the ring of $k$-valued algebraic functions on $I/\fn I$ by mapping $s_{i}$ to the $i$th coordinate function of the $k$-basis $\{[f_{1}],\cdots, [f_{c}]\}$ of of $I/\fn I$.

\begin{theorem}
\label{thm:closed}
Assume that $k$ is algebraically closed and that $I$ contains a regular element. Fix an integer $d\ge \depth P$.  Let $M$ be a finitely generated $R$-module and $J$ the annihilator of the $k[\bss]$-module 
\[
\Tor {d}{\wt Q}{k[\bss]}{M[\bss]}\oplus \Tor {d+1}{\wt Q}{k[\bss]}{M[\bss]}\,.
\]

If $\pdim_{P}M=\infty$, then $\mcV_{\pi}(M)=k^{c}$. Otherwise it is given by
\[
\mcV_{\pi}(M) = \{(a_{1},\dots,a_{c})\in k^{c}\mid h(\bsa)=0 \text{ for all polynomials } h(\bss)\in  J\}.
\]
In particular,  $\mcV_{\pi}(M)$ is a Zariski-closed subset of $k^{c}$.
\end{theorem}
 
 \begin{proof}
Set $p:=\pdim_{P}M$. If $\pdim_{P/(f)}M$ is finite for some regular element $f\in I$, then $p$ is finite as well. Therefore $\mcV_{\pi}(M)=k^{c}$ when $p=\infty$; see \ref{ssec:alternative}. For the rest of the proof we assume that $p$ is finite. 
 
Fix a point $\bsa\in\BA_{k}^{c}$. Let ${\bsa}'\in P^{c}$  be a lifting of $\bsa$, as in \ref{ssec:alternative}, and set $Q:=P/(f)$, where $f=\sum_{i} a'_{i}f_{i}$.  Let $G$ be a free resolution of $M[\bss]$ over $\wt Q$ and set 
\[
X:=k[\bss]\otimes_{\wt Q}G\,.
\]
By construction, this is a complex of free $k[\bss]$-modules and satisfies 
\[
\hh X\cong \Tor {}{\wt Q}{k[\bss]}{M[\bss]}\,.
\]
In particular, the $k[\bss]$-module $\HH iX$ is finitely generated for each $i$, and $0$ for $i<0$.

\setcounter{claim}{0}

\begin{claim}
\label{cl:closed-one}
Let $\eps_{\bsa}\colon k[\bss]\to k$ be the homomorphism of $k$-algebras  with $\eps_{\bsa}(x_{i})=a_{i}$ for each $i$. There is an isomorphism of graded $k$-vector spaces
\[
\Tor{}QkM \cong \Tor{}{k[\bss]}{k_{\bsa}}X \,.
\]
Indeed, by Proposition~\ref{prp:mx}(2) the complex $Q\otimes_{\wt Q} G$ of free $Q$-modules is a resolution of $M$, so $\Tor {}QkM$ is the homology of the complex 
 \[
 k\otimes_{Q}(Q\otimes_{\wt Q}G) \cong k\otimes_{\wt Q}G \cong k_{\bsa} \otimes_{k[\bss]}(k[\bss]\otimes_{\wt Q}G)=k_{\bsa}\otimes_{k[\bss]}X\,.
 \]
The second isomorphism holds because the composed map $\wt Q\to Q\to k$ factors as $\wt Q\to k[\bss]\xra{\eps_{\bsa}}k$.
Taking homology and keeping in mind that the $k[\bss]$-complex $X$ consists of free modules and is concentrated in non-negative degrees, one gets the stated isomorphism.
\end{claim}

\begin{claim}
\label{cl:two-closed}
Set $\fm_{\bsa}:=\Ker(\eps_{\bsa})$. One has
\[
\bsa \in \mcV_{\pi}(M) \iff \fm_{\bsa}\supseteq J\,.
\]
Indeed, consider the following chain of equivalences, where the first one holds by definition, and the second is by Claim~\ref{cl:closed-one}.
\begin{align*}
\bsa\not\in \mcV_{\pi}(M)
	& \iff \Tor iQkM=0 \text{ for $i\gg 0$} \\
	& \iff \Tor i{k[\bss]}{k_{\bsa}}X =0 \text{ for $i\gg 0$} \\
	& \iff {\HH iX}_{\fm_{\bsa}}=0 \text{ for $i\gg 0$.}
\end{align*}
The last one is given by  Lemma~\ref{lem:closed}, applied with $S:=k[\bss]$ and $\fp:=\fm_{\bsa}$. Since $\pdim_{R}M=p$ is finite, the Auslander-Buchsbaum Equality implies $d\ge p$. Proposition~\ref{prp:mx}(1) then gives the first equivalence below.
\begin{align*}
{\HH iX}_{\fm_{\bsa}}=0 \text{ for $i\gg 0$} 
	& \iff {\HH dX}_{\fm_{\bsa}} \oplus {\HH{d+1}X}_{\fm_{\bsa}}=0\\
	&\iff \fm_{\bsa}\not\supseteq J\,. 
\end{align*}
The second one is standard. Putting together the string of equivalences above yields the contrapositive of the desired claim. 
\end{claim}

Since $k$ is algebraically closed, Hilbert's Nullstellensatz implies that every maximal ideal of $k[\bss]$ is of the form $\fm_{\bsa}$, for some $\bsa\in k^{c}$, and that
\[
\fm_{\bsa}\supseteq J \iff h(\bsa)=0 \text{ for all polynomials } h(\bss)\in  J\,.
\]
Combining this with the conclusion in Claim~\ref{cl:two-closed} yields the desired result.
\end{proof}

\begin{remark}
\label{rem:complexes}
Theorem~\ref{thm:closed} carries over, with essentially the same proof, to the case when $M$ is an $R$-complex with $\hh M$ finitely generated. The crucial point is that, in the notation above, the natural morphism of $Q$-complexes
\[
Q\otimes_{\wt Q}^{\mathbf L} M[\bss] \lra M
\]
is a quasi-isomorphism. This can be verified using a spectral sequence as in the proof of Proposition~\ref{prp:mx}(2).
\end{remark}

\begin{remark}
\label{rem:history}
When $P$ is a domain the homological support set of $M$, as in \ref{ss:support}, is the one introduced by Jorgensen~\cite{Jo}, which extends the notion of varieties for modules over complete intersections due to Avramov~\cite{Av:vpd}. The last part of Theorem~\ref{thm:closed}--- that $\mcV_{\pi}(M)$ is closed---was proved by Jorgensen~\cite[Theorem~2.2]{Jo} under the additional assumption that $P$ is a domain. The argument in \emph{op.~cit.} is quite different from the one presented above, which builds on a idea in the proof of \cite[Theorem~3.1]{AB}.
\end{remark}

\section{Defining equations}
\label{sec:equations}
The main result of this section, Theorem~\ref{thm:equations}, gives equations that define the homological support set of modules over complete intersections. The statement, and its proof, involves some  linear algebra over commutative rings, recalled below.

\begin{remark}
\label{rem:minors}
Let $S$ be a local ring and $\delta\colon U \to V$ a homomorphism of finite free $S$-modules.  For an integer $r\ge 0$, let $I_{r}(\delta)$ denote the ideal generated by the $r\times r$ minors of a matrix representing $\delta$ in some bases for $U$ and $V$; the ideal $I_{r}(\delta)$ is independent of these choices. The following conditions are equivalent for any integer $r\ge 1$.
\begin{enumerate}[\quad\rm(1)]
\item 
$I_{r}(\delta) = S$;
\item
$\Image(\delta)$ shares with $V$  a free direct summand of rank $r$;
\item
$\Coker(\delta)$ can be generated, as an $S$-module, by $\rank_{S}V-r$ elements.
\end{enumerate}
Assume in addition that $S$ is a domain. With $S_{0}$ the field of fractions of $S$, the \emph{rank} of an $S$-module $M$ is the rank of the $S_{0}$-vector space $S_{0}\otimes_{S}M$. Then, when $\rank\Image(\delta)\le r$ the conditions above are equivalent to 
\begin{enumerate}[\quad\rm(4)]
\item
$\Coker(\delta)$ is free of rank equal to $\rank_{S}V-r$.
\end{enumerate}
These assertions are easy to verify, given that $I_{r}(\delta)=S$ if and only if there exists a choice of bases for $U$ and $V$ such that $\delta$ is represented by a matrix of the form 
\[
\begin{bmatrix} I_{r} & 0 \\ 0 & B \end{bmatrix}
\]
where $I_{r}$ is the identity matrix of size $r$; see also~\cite[Lemmas 1.4.8, 1.4.9]{BH}. 
\end{remark}

The result below, where $\sqrt{J}$ denotes the radical of an ideal $J$, concerns differential modules; cf.~\cite[Remark~1.6]{ABI}, and also \cite[Example 1.7]{ABI} that shows that the hypothesis that $S$ is regular is needed.

\begin{proposition}
\label{pr:dmodule}
Let $S$ be a regular ring and $\delta\colon S^{2r}\to S^{2r}$ an $S$-linear map with $\delta^{2}=0$. The annihilator $J$ of the $S$-module $\Ker(\delta)/\Image(\delta)$ satisfies 
\[
\sqrt{J}=\sqrt{I_{r}(\delta)}\,.
\]
\end{proposition}

\begin{proof}
Set $H:=\Ker(\delta)/\Image(\delta)$. The radical of an ideal is the intersection of the prime ideals containing it, and a prime ideal $\fp$ of  $S$ contains $J$ precisely when $H_{\fp}$ is non-zero. Thus, the desired statement is equivalent to: 
\[
H_{\fp}=0 \iff I_{r}(\delta)\not\subset \fp\,.
\]
Replacing $S$ by $S_{\fp}$ and $\delta$ by $\delta_{\fp}$, it thus suffices to prove that when $S$ is a regular local ring, one has $H=0$ if and only if $I_{r}(\delta)=S$.
 
There are exact sequences of $S$-modules
\begin{align}
\label{eq:KI}
&0\lra \Ker(\delta) \lra S^{2r}\lra \Image(\delta)\lra 0\,, \\ 
\label{eq:H}
&0\lra \Image(\delta) \lra \Ker(\delta) \lra H \lra 0\,, \\ 
\label{eq:HC}
&0\lra H \lra \Coker(\delta) \lra \Image(\delta)\lra 0\,.
\end{align}
In particular there are (in)equalities
\[
\rank_{S}\Ker(\delta) + \rank_{S}\Image(\delta)=2r \quad\text{and}\quad 
\rank_{S}\Ker(\delta) \ge \rank_{S}\Image(\delta) \,.
\]
Therefore one has 
\[
\rank_{S}\Ker(\delta)\ge r \ge  \rank_{S}\Image(\delta)\,,
\]
and both inequalities become equalities when $H=0$.

Assume $H=0$ so that $\rank_{S}\Image(\delta)=r$. Since the $S$-modules $\Coker(\delta)$ and $\Image(\delta)$ are isomorphic, by \eqref{eq:HC}, it remains to prove that $\Image(\delta)$ is free; then Remark~\ref{rem:minors} would yield $I_{r}(\delta)=S$.

For each integer $n\ge 1$ there is an exact sequence of $S$-modules
\[
0\lra \Image(\delta) \lra S^{2r} \lra \cdots \lra S^{2r}\lra \Image(\delta) \lra 0\,,
\]
of length $n$. In particular for $n=\dim S+2$, one sees that $\Image(\delta)$ is a $(\dim S)$-th syzygy module and hence free, as $S$ is a regular local ring. This is as desired.

Assume $I_{r}(\delta)=S$. Since $S$ is a domain and $\rank_{S}\Image(\delta)\le r$, it follows from Remark~\ref{rem:minors} that  $\Coker(\delta)$ is free of rank $r$. This implies that, with $l$ the residue field of $S$, the natural maps
\[
\Coker(\delta)\otimes_{S}l \xra{\ \cong\ } \Coker(\delta\otimes_{S}l) \quad\text{and} \quad
\Image(\delta)\otimes_{S}l \xra{\ \cong\ } \Image(\delta\otimes_{S}l)
\]
are isomorphisms. Thus, applying $(-)\otimes_{S}l$ to \eqref{eq:HC} and keeping in mind that $\Image(\delta)$ is free, one then gets an exact sequence
\[
0\to H\otimes_{S}l \to \Coker(\delta\otimes_{S}l) \to \Image(\delta\otimes_{S}l)\to 0\,,
\]
The isomorphisms above imply that both $l$-vector spaces on the right have rank $r$. It follows that $H\otimes_{S}l=0$; thus $H=0$, by Nakayama's Lemma.
\end{proof}

\subsection{Complete intersections}
\label{ssec:ci}
Let $(P,\fn,k)$ be a regular local ring of Krull dimension $d$ and containing its residue field, $k$, as a subring. Let $I$ be the ideal generated by a $P$-regular sequence $f_{1},\dots,f_{c}$ in $\fn^{2}$ and set $R:=P/I$.

As in \ref{ssec:closed}, let $P[\bss]$ be the polynomial ring over $P$ in indeterminates $\bss:=s_{1},\dots,s_{c}$,  set $\wt f:=\sum_{i} s_{i}f_{i}$ and $\wt Q:=P[\bss]/(\wt f)$.  Since $\wt Q$ is a hypersurface, the $\wt Q$-module $\wt Q/\fn \wt Q \cong k[\bss]$ has a free resolution $G$ with the property that $G_{i}\cong {\wt Q}^{2^{d-1}}$ and $\partial_{i}=\partial_{i+2}$ for all $i \ge c$ so that the complex $G_{\ge c}$ has the form
\[
\cdots \xra{\ B\ } G_{c+3} \xra{\ A\ } G_{c+2}\xra{\ B\ } G_{c+1}\xra{\ A\ } G_{c}
\]
The matrices $A$ and $B$ come from a matrix factorization of $\wt f$; see~\cite{Ei}.

Given an $R$-module $M$, the matrices $A$ and $B$ define $\wt Q$-linear endomorphisms of $M[\bss]^{2^{d-1}}$ and hence an endomorphism
\[
\begin{bmatrix}
 0 & A \\
 B & 0
 \end{bmatrix} \colon M[\bss]^{2^{d}}\lra M[\bss]^{2^{d}}\,.
\]
When $M$ has finite rank as a $k$-vector space, one can view this as an endomorphism of free $k[\bss]$-modules of rank $(\rank_{k}M)2^{d}$; we write $C(M)$ for this map.

\begin{theorem}
\label{thm:equations}
Let $R$ be as in \ref{ssec:ci}. If $M$ is an $R$-module with $\rank_{k}M$ finite, $\mcV_{R}(M)$ is  defined by the vanishing of the ideal $I_{r}(C(M))$, where $r:=(\rank_{k}M)2^{d-1}$.
\end{theorem}

\begin{proof}
Note that $\Tor{}{\wt Q}{k[\bss]}{M[\bss]}$ is the homology of the complex $G\otimes_{\wt Q}M[\bss]$, which for $i\geq c$ reads
\[
\cdots \lra M[\bss]^{2^{d-1}} \xra{\ B\ } M[\bss]^{2^{d-1}} \xra{\ A\ }  M[\bss]^{2^{d-1}}.
\]
Since $\depth P=d \ge c$, it follows that $\Tor{d+1}{\wt Q}{k[\bss]}{M[\bss]}\oplus  \Tor{d}{\wt Q}{k[\bss]}{M[\bss]}$ is the middle homology of the sequence
\[
M[\bss]^{2^{d}} \xra{\ C(M)\ } M[\bss]^{2^{d}} \xra{\ C(M)\ } M[\bss]^{2^{d}}
\]
of $k[\bss]$-modules. The desired statement is thus a consequence of Theorem~\ref{thm:closed}, and Proposition~\ref{pr:dmodule} applied with $S:=k[\bss]$ and $\delta:=C(M)$.
\end{proof}

\begin{remark}
Theorem~\ref{thm:equations} should be compared to \cite[Theorem~3.2]{AB} that describes equations defining $\mcV_{R}(M)$ in terms of data derived from a free resolution, say $F$, of $M$ over $P$, instead of the resolution of $k[\bss]$ over $\wt Q$. In detail: One can construct a system of higher homotopies, a family of endomorphism (not chain maps) of $F$ that encode the data that, because $M$ is an $R$-module, multiplication by $f_{1},\dots,f_{c}$ on $F$ is a commuting family of morphisms that are homotopic to zero. Then \cite[Theorem~3.2]{AB} expresses $\mcV_{\pi}(M)$ as the zero-set of appropriate minors of the matrix involving these higher homotopies.  In contrast, the $C(M)$ appearing in Theorem~\ref{thm:equations} is universal, in that it involves only a matrix factorization of $\wt f$, and matrices describing the action of the elements of $P$ on $M$, viewed as $k$-vector space.
\end{remark}

We finish by describing the free resolution of $k[\bss]$ over $\wt Q$ and, in particular, the matrices $A$ and $B$, when $R$ is a truncated power series ring; group algebras of abelian $p$-groups over a field of positive characteristic $p$ have this form.

\begin{example}
Let $k$ be a field, $P:=k[|t_{1},\dots,t_{d}|]$ the power series ring over $k$ in indeterminates $\bst:=t_{1},\dots,t_{d}$, and set
\[
R:=k[|t_{1},\dots,t_{d}|]/(t_{1}^{u_{1}}, \dots,t_{c}^{u_{c}})\,,
\]
where $u_{i}\ge 2$ for each $i$. 

Let $E$ be the Koszul complex over $\wt Q$ on $\bst$; thus $E=\wt Q\langle \bsx\mid \partial(x_{i})=t_{i}\rangle$.  The element $\sum_{i}s_{i}t_{i}^{u_{i}-1}x_{i}$ in $E_{1}$ is evidently a cycle.  Since $\bst$ is a regular sequence in $P[\bss]$ and $\wt f$ is a regular element contained in $(\bst)$, the complex
\[
E\big\langle y\mid \partial(y) = \sum_{i}s_{i}t_{i}^{u_{i}-1}x_{i}\big\rangle
\]
is a $\wt Q$-free resolution of $\wt Q/(\bst)$, that is to say, of $k[\bss]$; see Lemma~\ref{ssec:hyper}.  For any integer $n\ge c$ one has
\[
E\langle y\rangle_{n}:=
\begin{cases}
 \bigoplus_{i\ges 0} E_{2i} & \text{when $n$ is even}\\
 \bigoplus_{i\ges 0} E_{2i+1} & \text{when $n$ is odd}
\end{cases}
\]
In either parity, it is a finite free $\wt Q$-module of rank $2^{c-1}$. Moreover, after a suitable choice of bases, the differential on $E\langle y\rangle$ is given by 
\[
\partial_{\mathrm{even}}:= A \quad\text{and} \quad \partial_{\mathrm{odd}}:= B\,.
\]
with $A,B$ are square matrices  of size $2^{c-1}$, and coefficients in $\wt Q$, described below.

The rows and columns of $A$ and $B$ are indexed by subsets of the sequences 
\[
U:=\{\bsh=(h_{1},\dots,h_{c})\mid \bsh \in \{0,1\}^{c}\}
\]
ordered weighted lexicographically.

The matrix $A$ has rows indexed by sequences $\bsh$ in $U$ with  $\sum h_{i}$ odd and columns indexed by sequences $\bsj$ with $\sum j_{i}$ even. The entry in $A$ in position $(\bsh,\bsj)$ is  $0$ when  $\sum_{i} |j_{i}-h_{i}| \ne 1$; otherwise, there is a unique integer $p$ such that $|j_{p}-h_{p}|=1$, and setting $n=\sum_{i<p}j_{i}$, one has
\[
A_{\bsh,\bsj}:=
\begin{cases}
(-1)^{n}s_{p}t_{p}^{u_{p}-1} & \text{if $h_{p}=1$ and $j_{p}=0$}\\
(-1)^{n} t_{p} & \text{if $h_{p}=0$ and $j_{p}=1$}
\end{cases}
\]
The rows of $B$ are indexed by sequences $\bsh$ with  $\sum h_{i}$ even and the columns are indexed by sequences $\bsj$ with $\sum j_{i}$ odd. Its entries are defined exactly as for $A$.
\end{example}

\section{Group algebras of elementary abelian groups}
\label{sec:groups}
Let $k$ be a field of positive characteristic $p$, and $P:=k[\bss]$ the polynomial ring in indeterminates $\bss:=s_{1},\dots,s_{c}$. Set $I:=(s_{1}^{p},\dots, s_{c}^{p})$ and $R:=P/I$. The $k$-algebra $R$ is the group algebra of an elementary abelian $p$-group and, more generally, of a quasi-elementary finite group scheme~\cite{FP}. The ring $R$ admits a coproduct making it into a Hopf algebra, but this structure plays no role in what follows.

Set $\fn:=(\bss)$; this is a maximal ideal of $P$. Each $u \in \fn$ yields a homomorphism
\begin{align*}
&\eta_{u}\colon k[t]/(t^{p}) \lra R \quad\text{such that $\eta_{u}(t) = u \mod I$.} 
\intertext{For such a $u$ one has $u^{p}\in I$ and hence also a surjective homomorphism}
 &\pi_{u^{p}}\colon P/(u^{p})\lra R
\end{align*}
of $k$-algebras; see Section~\ref{sec:Hypersurfaces}. In what follows it is expedient to write $M\da{u}$ for $M$ viewed as a $k[t]/(t^{p})$-module via $\eta_{u}$. The proof of the next result extracts an argument from that for \cite[Theorem 3.2]{AB}, and is given after Lemma~\ref{lem:polynomial}.

\begin{theorem}
\label{thm:groups}
For any $R$-module $M$ one has 
\[
M\da{u} \text{is free} \iff \fdim_{P/(u^{p})} (M)  < \infty\,.
\]
\end{theorem}

In the equivalence above, the invariant on the right is the flat dimension.

\subsection{Flat dimension}
\label{ss:fdim}
Let $Q$ be a commutative noetherian ring and $N$ a $Q$-module. Recall that $\fdim_{Q}(N)$, the \emph{flat dimension} of $N$, is the length of the shortest flat resolution of $N$ over $Q$. Evidently $\fdim_{Q}(N)\le \pdim_{Q}(N)$; equality holds when $N$ is finitely generated. 

The $Q$-module $N$ is \emph{$J$-torsion}, for some ideal $J$ in $Q$, if for each $u \in N$ there is an integer $l$ such that $J^{l} u=0$. If $N$ is $\fn$-torsion for a maximal ideal $\fn$ of $Q$, then
\begin{equation}
\label{eq:fdim}
\fdim_{Q}N <\infty \iff \Tor iQ{Q/\fn}N=0 \quad \text{for $i\gg 0$}\,.
\end{equation}
This follows form \cite[Propositions 5.3.F]{AF}, keeping in mind that for any prime ideal $\fq$ of $Q$, when $\fq\ne \fn$, one has $\Tor{}Q{(Q/\fq)_{\fq}}{N_{\fq}}=0$.

We  record a few consequences of Theorem~\ref{thm:groups}. 

\begin{corollary}
If an element $w\in \fn$ satisfies $u-w\in\fn^{2}$, then $M\da{w}$ is free if and only if is $M\da{u}$.
\end{corollary}

\begin{proof}
Evidently $\fn^{l}$ annihilates $R$, and hence  $M$, for $l\gg 0$; thus $M$ is $\fn$-torsion. Given \eqref{eq:fdim}, the desired result follows from Theorems~\ref{thm:hypersurfaces}(2) and \ref{thm:groups}.
\end{proof}

\subsection{Rank Varieties}
In this paragraph we assume that $k$ is algebraically closed. Let $\BA_{k}^{c}$ be the affine space over $k$, of dimension $c$. We write $\bsa$ for a point $(a_{1},\dots,a_{c})$ in this space. For each $R$-module $M$ set
\[
\rvar R(M):=\left\{\bsa\in \BA_{k}^{c} \left| M\da{u} \text{is not free for $u=\sum_{i}a_{i}s_{i}$}\right.\right\}\cup\{0\}\,.
\]
This is called the \emph{rank variety} of $M$. It was introduced by Carlson who also proved that it is a closed subset of $\BA_{k}^{c}$; see~\cite[\S4]{Ca}.  The result below, which relates the rank variety of $M$ to its homological support set in the sense of \ref{ss:support}, is contained in  \cite[Theorem~7.5]{Av:vpd}. The argument in \emph{op.~cit.} gets to the stated bijection via the cohomological variety of $M$. Our proof, using  Theorems~\ref{thm:hypersurfaces} and \ref{thm:groups}, is more direct.

\begin{corollary}
The map  $\BA_{k}^{c}\to \BA_{k}^{c}$ sending $(a_{1},\dots,a_{c})$ to $(a^{p}_{1},\dots,a^{p}_{c})$ is bijective and 
for each $R$-module $M$ it restricts to a bijection
\[
\rvar R(M) \xra{\ \cong\ } \var R(M)\,.
\]
\end{corollary}

\begin{proof}
The given map is a bijection on $\BA_{k}^{c}$ because $k$ is algebraically closed. The assertion about the varieties associated to $M$ is a translation of Theorem~\ref{thm:groups}.
\end{proof}

A classical observation---see~\cite[Part III]{Kap}---will be used in the proof of Theorem~\ref{thm:groups}. It is a special case of a statement that holds for  smooth morphisms.

\subsection{Polynomial extensions}
  \label{ssec:polynomial}
Let $S$ be a commutative ring, $S[\bsx]$ a polynomial ring in $d$ indeterminates $\bsx$, and $M$ a $S[\bsx]$-module.

   \begin{lemma}
  \label{lem:polynomial}
The following inequalities hold:
\[
\fdim_{S} M \le \fdim_{S[\bsx]} M \le \fdim_{S} M + d\,.
\]
\end{lemma}

  \begin{proof}
The inequality on the left holds because any flat resolution over $S[\bsx]$ is also a flat resolution over $S$,
due to the freeness of $S[\bsx]$ as a $S$-module.

For the inequality on the right, we may assume that $\fdim_{S} M$ is finite.  Induction on $d$ then shows that it suffices 
to deal with the case when $S[x]$ is a polynomial ring in a single indeterminate. In the ring  $S[x]\otimes_{S}S[x]$, the element 
$x\otimes 1 - 1\otimes x$ is regular, so we have an exact sequence of $S[x]\otimes_{S}S[x]$-modules 
\[
0\lra (S[x]\otimes_{S}S[x]) \xra{\ x\otimes 1 - 1\otimes x\ }  (S[x]\otimes_{S}S[x]) \lra S[x] \to 0
\]
It splits as a sequence of right $S[x]$-modules, so $(-\otimes_{S[x]}M)$ yields an exact sequence
\[
0\lra (S[x]\otimes_{S}M) \xra{\ x\otimes 1 - 1\otimes x\ }  (S[x]\otimes_{S}M) \lra M \to 0
\]
of $S[x]$-modules, with $S[x]$ acting from the left. The desired inequality follows.
  \end{proof}

\begin{proof}[Proof of Theorem~\ref{thm:groups}]
Assume $u$ is in $\fn^{2}$. Then $\eta_{u}(t^{\lceil \frac{p+1}2\rceil})=0$ so no $R$-module is free as an $k[t]/(t^{p})$ module. Moreover, since $(a+b)^{p}=a^{p}+b^{p}$ for all $a,b\in P$,  one has  $u^{p}\in \fn I$. Thus Theorem~\ref{thm:hypersurfaces}(3) yields $\fdim_{P/(u^{p})} (M)=\infty$. 

In the rest of the proof we assume $u$ lies in $\fn\setminus \fn^{2}$. Consider the commutative diagram of homomorphisms of $k$-algebras
\[
\xymatrixcolsep{1.5pc}
\xymatrixrowsep{1.2pc}
\xymatrix{
	&\frac{k[\bss]}{(u^{p})}  \ar@{->>}^{\pi_{u^{p}}}[dr]&   \\
\frac{k[t]}{(t^{p})} \ar@{->}[ur]^-{t\mapsto u} \ar@{->}[rr]^-{\ \eta_{u} \ }& & R	}
\]
The condition on $u$ implies that the map $k[t]/(t^{p})\lra k[\bss]/(u^{p})$ is a polynomial extension. Thus, from Lemma \ref{lem:polynomial} one gets the first equivalence below:
\[
\begin{aligned}
\fdim_{P/(u^{p})} (M) <\infty  
	&\iff \fdim(M\da{u})<\infty \\  	
	&\iff M\da{u } \text{ is free}
\end{aligned}
\]
The second equivalence holds because $k[t]/(t^{p})$ is artinian.
\end{proof}

The preceding results can be reinterpreted as statements about $p$-nilpotent operators on $k$-vector spaces.

\subsection{Nilpotent operators}
\label{sec:nilpotent}
Let $V$ be a $k$-vector space and $\alpha\colon V\to V$ a $k$-linear map such that $\alpha^{n}=0\ne\alpha^{n-1}$ holds for some integer $n\ge1$.  Evidently, one then has $\alpha V\subseteq \Ker(\alpha^{n-1})$; when equality holds we say that  $\alpha$ has \emph{maximal image}. We need an elementary lemma concerning the standard correspondence between linear operators  $\alpha\colon V\to V$ of nilpotency degree at most $n$ and $k[t]/(t^{n})$-module structures on $V$. The proof is straightforward linear algebra, and so is omitted.

\begin{lemma}
\label{lem:nilpotent}
For a $k$-linear map $\alpha\colon V\to V$ with $\alpha^{n}=0\ne\alpha^{n-1}$ for some integer $n\ge1$, 
the following conditions are equivalent.
  \begin{enumerate}[\quad\rm(1)]
    \item
$\alpha V= \Ker(\alpha^{n-1})$.
    \item
$\alpha^{n-1}V= \Ker(\alpha)$. 
    \item
$\alpha^i V= \Ker(\alpha^{n-i})$ for $i=1,\dots,n-1$.
    \item
$V$ is free as a $k[x]/(x^n)$-module. \qed
  \end{enumerate}
\end{lemma}

The next theorem is due to Friedlander and Pevtsova~\cite[Proposition~2.2]{FP}, extending work of  Carlson~\cite[Lemma~6.4]{Ca} and  Bendel,  Friedlander, and Suslin~\cite[Lemma~6.4]{SFB}. It is a key ingredient in the theory of $\pi$-points for finite groups schemes, which subsumes the theory of rank varieties for elementary abelian $p$-groups. 

\begin{theorem}
\label{thm:operators}
Let $k$ be a field of characteristic $p>0$, let $V$ be a $k$-vector space, and let $\alpha,\beta,\gamma\colon V\to V$ 
be commuting $k$-linear maps that satisfy
\[
\alpha^{p} = 0 \quad\text{and}\quad \beta^{p} = 0\,.
\]
Then $\alpha$ has maximal image if and only if $\alpha + \beta\gamma$ has maximal image. 
\end{theorem}

\begin{proof}
Form the $k$-algebras
\[
P:=k[a,b,c] \quad\text{and}\quad R:= \frac{P}{(a^{p},b^{p})}
\]
Due to the hypotheses one has an $R$-module structure on $V$, with $a,b$, and $c$ acting on $V$ via $\alpha,\beta$, and $\gamma$ respectively. Consider the homomorphism of $k$-algebras
\[
\begin{gathered}
\xymatrixcolsep{2pc}
\xymatrix{
\frac{k[t]}{(t^{p})} \ar@<.5ex>[r]^{\sigma_{a}} \ar@<-.5ex>[r]_{\sigma_{a+bc}} & R
}
\end{gathered}
\quad \text{where $\sigma_{a}(t) = a$ and $\sigma_{a+bc}(t) = a+bc$}.
\]
In view of Lemma~\ref{sec:nilpotent} the desired result is that the module $V\da{a}$ is free if and only if so is $V\da{a+bc}$. This follows from the chain of equivalences below where the first and the last ones hold by Theorem~\ref{thm:groups}.
\begin{align*}
V\da{a} \text{is free} 
	&\iff \fdim_{P/(a^{p})} (V) < \infty \\
	&\iff \fdim_{P/(a^{p}+b^{p}c^{p})} (V)  < \infty \\
	&\iff V\da{a+bc} \text{is free} 
\end{align*}
In view of \eqref{eq:fdim}, the one in the middle is given by Theorem~\ref{ssec:hyper}(2), since $b^{p}c^{p}$ lies in $(a,b,c)(a^{p},b^{p})$.
\end{proof}

\end{document}